\documentclass[12pt]{article}%
\usepackage{amssymb}
\usepackage{amsmath}
\usepackage{amsfonts}
\usepackage{graphicx}
\usepackage{hyperref}%
\newtheorem{theorem}{Theorem}[section]

\newtheorem{corollary}[theorem]{Corollary}

\newtheorem{definition}[theorem]{Definition}

\newtheorem{lemma}[theorem]{Lemma}

\newenvironment{proof}[1][Proof]{\noindent\textbf{#1.} }{\ \rule{0.5em}{0.5em}}
\begin{document}

\title{A Johnson-Kist type representation for truncated vector lattices}
\author{Karim Boulabiar\thanks{Corresponding author:
\texttt{karim.boulabiar@fst.utm.tn}}\quad and\quad Rawaa Hajji\medskip\\{\small Laboratoire de Recherche LATAO}\\{\small D\'{e}partement de Mathematiques, Facult\'{e} des Sciences de Tunis}\\{\small Universit\'{e} de Tunis El Manar, 2092, El Manar, Tunisia}}
\date{}
\maketitle

\begin{abstract}
We introduce the notion of (maximal) multi-truncations on a vector lattice as
a generalization of the notion of truncations, an object of recent origin. We
obtain a Johnson-Kist type representation of vector lattices with maximal
multi-truncations as vector lattices of almost-finite extended-real continuous
functions. The spectrum that allow such a representation is a particular set
of prime ideals equipped with the hull-kernel topology. Various
representations from the existing literature will appear as special cases of
our general result.

\end{abstract}



\section{Introduction}

In his paper \cite{S}, Stone investigates the extent to which certains linear
functionals on function spaces can be representated as integrals for some
measures. It turnd out that the suitable spaces for such representations to
hold is the function vector lattices $L$ closed under meat with the constant
function $1$, i.e., $1\wedge f\in L$, for all $f\in L$. The classical
Daniel-Stone Representation Theorem (see, e.g., Chapter V in \cite{K}) is a
remarkable illustration of the importance of these spaces in Measure Theorey.
Fremlin, in his book \cite{F74}, speaks about \textsl{truncated vector
lattices} of functions (see also \cite{F06}). In the present paper, we shall
follow the Fremlin terminology.

We call a \textsl{truncation} on a (real) vector lattice $L$ any (non-zero)
function that takes each positive vector $f$ of $L$ to a positive vector
$f^{\ast}$ of $L$, and has the following properties:

\begin{enumerate}
\item[$\left(  \mathfrak{a}\right)  $] If $0\leq f,g\in L$ then $f\wedge
g^{\ast}\leq f^{\ast}\leq f$.

\item[$\left(  \mathfrak{b}\right)  $] If $0\leq f\in L$ and $\left(
nf\right)  ^{\ast}=nf$ for all $n\in\left\{  1,2,...\right\}  $, then $f=0$.
\end{enumerate}

\noindent By a \textsl{truncated vector lattice} we mean a vector lattice
along with a truncation. If in addition the condition

\begin{enumerate}
\item[$\left(  \mathfrak{w}\right)  $] $0\leq f\in L$ and $f^{\ast}=0$ imply
$f=0$
\end{enumerate}

\noindent is fulfilled, we speak about a \textsl{weak truncation} on $L$ and a
\textsl{weakly truncated vector lattice }$L$ (this would explain the letter
$\mathfrak{w}$ we use to label this third axiom).

In his pioneer works \cite{B1,B2}, Ball extended the classical Yosida
Representation of Archimedean vector lattices with weak units to the wider
class of Archimedean weakly truncated vector lattices (Archimedean truncs in
Ball's terminology). In this regard, he got a beautiful representation of
Archimedean truncs as vector lattices of extended-real valued continuous
functions. To be more precise, let $C^{\infty}\left(  \mathcal{X}\right)  $
denote the set of all almost-finite extended-real continuous valued functions
on a topological space $\mathcal{X}$. Ball proves that for any Archimedean
weakly truncated vector lattice $L$, there exists a locally compact Hausdorff
space $\mathcal{X}$ such that \textbf{(i) }$L$ is (lattice isomorphic with) a
vector lattice of functions in $C^{\infty}\left(  \mathcal{X}\right)  $,
\textbf{(ii) }$L$ separates points from closed sets in $\mathcal{X}$, and
\textbf{(iii) }the truncation on $L$ is provided by meet with the constant
function one. A copy of the spectrum $\mathcal{X}$ in Ball's result is a
Tychonoff product of spaces of truncation preserving lattice homomorphisms
with domain $L$ and codomains totally ordered truncated vector lattices.

Our main purpose in this paper is to extended the Ball Representation Theorem
to truncated vector lattices (without assuming the condition $\left(
\mathfrak{w}\right)  $) using a Johnson-Kist type approach (see \cite{JK}),
which is based on the hull-kernel topology on prime ideals rather than the
Tychonoff product topology on homomorphisms. Indeed, our central result
stipulates that, given an Archimedean truncated vector lattice $L$, a locally
compact Hausdorff space $\mathcal{X}$ can be found such that \textbf{(i) }$L$
is identified with a vector lattice of functions in $C^{\infty}\left(
\mathcal{X}\right)  $, \textbf{(ii) }$L$ separates points from closed sets in
$\mathcal{X}$, \textbf{(iii) }all `functions' in $L$ vanish at infinity, and
\textbf{(iv) }the truncation on $L$ is provided by meet with the
characteristic function $1_{\mathcal{Y}}$ of some open-closed set
$\mathcal{Y}$ in $\mathcal{X}$. This representation has been made possible
through the notion of multi-truncations, which is a `new' concept involving
(maximal) family of truncations on vector lattices. As a matter of fact, a
\textsl{multi-truncation} $\mathcal{T}$ on the vector lattice $L$ is just a
family of truncations on $L$ with disjoint (i.e., orthogonal) ranges. The idea
is to see that the given truncation of the truncated vector lattice $L$ is
contained in a maximal multi-truncation $\mathcal{T}$ on $L$. It then turns
out that the set%
\[
\mathcal{X}=%
{\displaystyle\bigcup_{\ast\in\mathcal{T}}}
{\displaystyle\bigcup_{f\in L^{+}}}
\mathrm{Val}\left(  f^{\ast}\right)
\]
along with its hull-kernel topology is a suitable representation spectrum.
Here, $\mathrm{Val}\left(  f^{\ast}\right)  $ denotes the collection of all
prime ideals of $L$ which are maximal with respect to not containing $f^{\ast
}$. As a first application, it will follow that $L$ is a weakly truncated
vector lattice if and only if $\mathcal{X}=\mathcal{Y}$, which gives the Ball
Representation Theorem referred to above. We shall then focus on the
representation of a \emph{strongly truncated vector lattice}, that is, a
truncated vector lattice $L$ for which the extra condition

\begin{enumerate}
\item[$\left(  \mathfrak{s}\right)  $] if $0\leq f\in L$ then $\left(
nf\right)  ^{\ast}=nf$ for some $n\in\left\{  1,2,...\right\}  $
\end{enumerate}

\noindent is satisfied. Actually, we shall use our previous result (as well as
a Stone-Weierstrass Approximation Type Theorem) to show that any Archimedean
strongly truncated vector lattice $L$ can be embedded as a uniformly dense
vector sublattice of the Banach lattice $C_{0}\left(  \mathcal{X}\right)  $ of
all real-valued continuous functions on $\mathcal{X}$ vanishing at infinity.
Moreover, since $L$ is, clearly, a weakly truncated vector lattice, this
embedding preserves truncation and its range still separates points from
closed sets in $\mathcal{X}$. This latter representation has been obtained in
\cite{BH} in a completely different way.

Finally, we suggest the reader keeps the textbook \cite{LZ} by Luxemburg and
Zaanen within arm's reach, so it is on hand whenever he needs more information
on vector lattices and (order) ideals.

\section{Tools on truncations and prime ideals}

This section presents the basic properties of truncations which are relevant
to our development.

\begin{quote}
\textsl{In order to avoid unnecessary repetition, }$L$ \textsl{stands
throughout the paper for a vector lattice with }$L^{+}$ \textsl{as positive
cone.}
\end{quote}

\noindent We start our investigation with the central definition of this paper.

\begin{definition}
\label{trunc}\emph{By a }truncation\emph{\ on }$L$\emph{\ is meant a nonzero
function that takes each positive vector }$f\ $\emph{of }$L$\emph{\ to a
positive vector }$f^{\ast}$\emph{\ of }$L$\emph{\ and has the following
properties:}

\begin{enumerate}
\item[$\left(  \mathfrak{a}\right)  $] $f\wedge g^{\ast}\leq f^{\ast}\leq f$
\emph{for all }$f,g\in L^{+}$.

\item[$\left(  \mathfrak{b}\right)  $] \emph{If }$f\in L^{+}$ \emph{and}
$\left(  nf\right)  ^{\ast}=nf$ \emph{for all} $n\in\left\{  1,2,...\right\}
$ \emph{then} $f=0$.
\end{enumerate}
\end{definition}

\noindent It is readily checked that the condition $\left(  \mathfrak{a}%
\right)  $ is met when and only when the condition

\begin{enumerate}
\item[$\left(  \mathfrak{a}^{\prime}\right)  $] $f\wedge g^{\ast}=f^{\ast
}\wedge g$ for all $f,g\in L^{+}$
\end{enumerate}

\noindent is verified. By the way, we shall be free to use, depending on the
context, either the condition $\left(  \mathfrak{a}\right)  $ or the
equivalent version $\left(  \mathfrak{a}^{\prime}\right)  $.

\noindent We record some algebraic identities which will come often in handy.

\begin{lemma}
\label{Birkhoff}If $\ast$ is a truncation on $L$ and $f,g\in L^{+}$ then

\begin{enumerate}
\item[\emph{(i)}] $\left(  f\wedge g\right)  ^{\ast}=f^{\ast}\wedge g^{\ast}$,

\item[\emph{(ii)}] $\left(  f\vee g\right)  ^{\ast}=f^{\ast}\vee g^{\ast}$, and

\item[\emph{(iii)}] $\left\vert f^{\ast}-g^{\ast}\right\vert \leq\left\vert
f-g\right\vert ^{\ast}$.
\end{enumerate}
\end{lemma}

\begin{proof}
$\mathrm{(i)}$ It is easily seen that%
\[
\left(  f\wedge g\right)  ^{\ast}\leq f^{\ast}\quad\text{and}\quad\left(
f\wedge g\right)  ^{\ast}\leq g^{\ast}.
\]
Therefore,%
\[
\left(  f\wedge g\right)  ^{\ast}\leq f^{\ast}\wedge g^{\ast}\leq f\wedge
g\wedge g^{\ast}=\left(  f\wedge g\right)  ^{\ast}\wedge g\leq\left(  f\wedge
g\right)  ^{\ast},
\]
which gives the first equality.

$\mathrm{(ii)}$ Since $f^{\ast}\leq f$ and $g^{\ast}\leq g$, we get%
\begin{align*}
f^{\ast}\vee g^{\ast}  &  =\left(  f\vee g\right)  \wedge\left(  f^{\ast}\vee
g^{\ast}\right)  =\left(  \left(  f\vee g\right)  \wedge f^{\ast}\right)
\vee\left(  \left(  f\vee g\right)  \wedge g^{\ast}\right) \\
&  =\left(  \left(  f\vee g\right)  ^{\ast}\wedge f\right)  \vee\left(
\left(  f\vee g\right)  ^{\ast}\wedge g\right)  =\left(  f\vee g\right)
^{\ast}\wedge\left(  f\vee g\right)  =\left(  f\vee g\right)  ^{\ast}.
\end{align*}
This is the required equality.

$\mathrm{(iii)}$ The classical Birkhoff's Inequality (see, e.g., \cite[Theorem
1.9]{AB}) allows us to write%
\begin{align*}
\left\vert f^{\ast}-g^{\ast}\right\vert  &  =\left\vert f\wedge\left(  f\vee
g\right)  ^{\ast}-g\wedge\left(  f\vee g\right)  ^{\ast}\right\vert \\
&  \leq\left\vert f-g\right\vert \wedge\left(  f\vee g\right)  ^{\ast}%
\leq\left\vert f-g\right\vert ^{\ast}.
\end{align*}
This completes the proof of the lemma.
\end{proof}

The next lines deal with prime ideals on a vector lattice along with a
truncation. Recall here that a vector subspace $P$ of the vector lattice $L$
is called an \emph{ideal} of $L$ if $P$ contains with any vector $f\in P$ all
vectors $g\in L$ such that $\left\vert g\right\vert \leq\left\vert
f\right\vert $. The ideal $P$ of $L$ is said to be \emph{prime} if%
\[
f,g\in L\text{ and }f\wedge g\in P\quad\text{imply\quad}f\in P\text{ or }g\in
P.
\]
It is not hard to prove that the ideal $P$ is prime if and only if, for every
$f,g\in L$ with $f\wedge g=0$, either $f\in P$ or $g\in P$. Also, a necessary
and sufficient condition for the ideal $P$ of $L$ to be prime is that
$f^{+}\in P$ or $f^{-}\in P$ for any $f\in L$.

\begin{quote}
\textsl{The set of all proper prime ideals on }$L$\textsl{\ will be denoted
by} $\mathcal{P}$.
\end{quote}

\noindent Chapter 5 in \cite{LZ} contains a thorough study of prime ideals on
a vector lattice to which we refer the reader for more information on the subject.

We now list features of prime ideals which will be of great use later.

\begin{lemma}
\label{Ball}Let $\ast$ be a truncation on $L$, $P\in\mathcal{P}$, and $f,g\in
L^{+}$.

\begin{enumerate}
\item[\emph{(i)}] If $f-f^{\ast}\notin P$ then $g^{\ast}-g\wedge f^{\ast}\in
P$,

\item[\emph{(ii)}] If $\left(  f-f^{\ast}\right)  ^{\ast}\wedge\left(
g-g^{\ast}\right)  ^{\ast}\notin P$ then $f^{\ast}-g^{\ast}\in P$, and

\item[\emph{(iii)}] $f^{\ast}\in P$ if and only if $\left(  nf\right)  ^{\ast
}\in P$ for all $n\in\left\{  1,2,...\right\}  $.
\end{enumerate}
\end{lemma}

\begin{proof}
$\mathrm{(i)}$ Clearly, $f^{\ast}\wedge g=f^{\ast}\wedge g^{\ast}$ and so%
\[
g^{\ast}-g\wedge f^{\ast}=g^{\ast}-g^{\ast}\wedge f^{\ast}=\left(  g^{\ast
}-f^{\ast}\right)  ^{+}.
\]
Hence,%
\[
\left(  g^{\ast}-g\wedge f^{\ast}\right)  \wedge\left(  f-f^{\ast}\right)
=\left(  g^{\ast}-f^{\ast}\right)  ^{+}\wedge\left(  f-f^{\ast}\right)
=\left(  \left(  g^{\ast}\wedge f\right)  -f^{\ast}\right)  ^{+}=0.
\]
As $P$ is a prime and $f-f^{\ast}\notin P$, we get $g^{\ast}-g\wedge f^{\ast
}\in P$ and $\mathrm{(i)}$ follows.

$\mathrm{(ii)}$ Applying $\mathrm{(i)}$ twice, we see that%
\[
f^{\ast}-f\wedge g^{\ast}\in P\quad\text{and}\quad g^{\ast}-g\wedge f^{\ast
}\in P.
\]
As $g\wedge f^{\ast}=g^{\ast}\wedge f$, we obtain%
\[
f^{\ast}-g^{\ast}=\left(  f^{\ast}-f\wedge g^{\ast}\right)  -\left(  g^{\ast
}-g\wedge f^{\ast}\right)  \in P,
\]
as desired.

$\mathrm{(iii)}$ The `if' part is obvious. The `only if' will be established
by induction. We have nothing to show for $n=1$. Hence, let $n\geq2$ and use
Lemma \ref{Birkhoff} $\mathrm{(iii)}$ to get%
\[
0\leq\left(  nf\right)  ^{\ast}-f^{\ast}\leq\left(  \left(  n-1\right)
f\right)  ^{\ast}.
\]
Thus,%
\[
0\leq\left(  nf\right)  ^{\ast}\leq f^{\ast}+\left(  \left(  n-1\right)
f\right)  ^{\ast}%
\]
and the proof is complete.
\end{proof}

To make the notation less cluttered, if $P\in\mathcal{P}$ and $\ast$ is a
truncation on $L$, we put%
\[
\pi^{\ast}\left(  P\right)  =\left\{  u\in L^{+}:\left(  u-u^{\ast}\right)
^{\ast}\notin P\right\}  .
\]
Fundamental properties of these sets are given in the last lemma of these preliminaries.

\begin{lemma}
\label{BH}Let $\ast$ be a truncation on $L$ and $f\in L$.

\begin{enumerate}
\item[\emph{(i)}] There exists $n\in\left\{  1,2,...\right\}  $ such that, if
$P\in\mathcal{P}$, then%
\[
f^{\ast}\notin P\quad\text{if and only if\quad}nf\in\pi^{\ast}\left(
P\right)  .
\]

\item[\emph{(ii)}] If $P\in\mathcal{P}$ and $u,v\in\pi^{\ast}\left(  P\right)
$, then%
\[
\left\{  \alpha\in\mathbb{R}:\left(  f-\alpha u^{\ast}\right)  ^{+}\in
P\right\}  =\left\{  \alpha\in\mathbb{R}:\left(  f-\alpha v^{\ast}\right)
^{+}\in P\right\}  .
\]

\end{enumerate}
\end{lemma}

\begin{proof}
$\mathrm{(i)}$ If $f^{\ast}\notin P$ then $f>0$. The condition $\left(
\mathfrak{b}\right)  $ ensures the existence of $m\in\left\{  1,2,...\right\}
$ such that $\left(  mf\right)  ^{\ast}<mf$. Put $n=2m$ and take
$P\in\mathcal{P}$. Assume that $nf\in\pi^{\ast}\left(  P\right)  $, that is,
$\left(  nf-\left(  nf\right)  ^{\ast}\right)  ^{\ast}\notin P$. Since,%
\[
0\leq\left(  nf-\left(  nf\right)  ^{\ast}\right)  ^{\ast}\leq\left(
nf\right)  ^{\ast}=\left(  nf\right)  ^{\ast}\wedge nf\leq n\left(  \left(
nf\right)  ^{\ast}\wedge f\right)  \leq nf^{\ast},
\]
we get $f^{\ast}\notin P$. Conversely, suppose that $f^{\ast}\notin P$. For
brevity, put $g=mf$ and notice that $0<g^{\ast}<g$. If $g-g^{\ast}\in P$ then
a prime ideal $Q$ of $P$ (and so of $L$) can be found so that $g-g^{\ast
}\notin Q$ (see, e.g., \cite[ Theorem 33.5]{LZ}). Hence, in any case, there
exists a prime ideal $Q$ of $L$ such that%
\[
Q\subset P\text{\quad and\quad}g-g^{\ast}\notin Q
\]
(take $Q=P$ if $g-g^{\ast}\notin P$). If we apply Lemma \ref{Ball}
$\mathrm{(i)}$ with $Q$ and $2g$, we obtain%
\[
\left(  2g\right)  ^{\ast}-g^{\ast}=\left(  2g\right)  ^{\ast}-\left(
g^{\ast}\wedge2g\right)  \in Q\subset P.
\]
So, as $g^{\ast}\notin P$, we get%
\[
2g^{\ast}-\left(  2g\right)  ^{\ast}=g^{\ast}-\left(  \left(  2g\right)
^{\ast}-g^{\ast}\right)  \notin P.
\]
Whence,%
\[
\left(  2g^{\ast}-\left(  2g\right)  ^{\ast}\right)  ^{\ast}\wedge g=\left(
2g^{\ast}-\left(  2g\right)  ^{\ast}\right)  \wedge g^{\ast}\notin P,
\]
from which we infer that%
\[
\left(  2g^{\ast}-\left(  2g\right)  ^{\ast}\right)  ^{\ast}\notin P.
\]
On the other hand, from Lemma \ref{Birkhoff} $\mathrm{(iii)}$ it follows that
\[
0\leq\left(  2g\right)  ^{\ast}=\left(  2g\right)  ^{\ast}-g^{\ast}+g^{\ast
}\leq\left(  2g-g\right)  ^{\ast}+g^{\ast}=2g^{\ast}.
\]
Thus,%
\[
0\leq\left(  2g^{\ast}-\left(  2g\right)  ^{\ast}\right)  ^{\ast}\leq\left(
2g-\left(  2g\right)  ^{\ast}\right)  ^{\ast}.
\]
Accordingly,%
\[
\left(  nf-\left(  nf\right)  ^{\ast}\right)  ^{\ast}=\left(  2g-\left(
2g\right)  ^{\ast}\right)  ^{\ast}\notin P,
\]
and $\mathrm{(i)}$ follows.

$\mathrm{(ii)}$ Let $u,v\in\pi\left(  P\right)  $ and $\alpha\in\mathbb{R}$.
Hence, $u^{\ast}-v^{\ast}\in P$ (where we use Lemma \ref{Ball} $\mathrm{(ii)}%
$), so if $\left(  f-\alpha u^{\ast}\right)  ^{+}\in P$ then%
\[
f-\alpha v^{\ast}=f-\alpha u^{\ast}+\alpha\left(  u^{\ast}-v^{\ast}\right)
\leq\left(  f-\alpha u^{\ast}\right)  ^{+}+\alpha\left(  u^{\ast}-v^{\ast
}\right)  \in P.
\]
It follows that $\left(  f-\alpha v^{\ast}\right)  ^{+}\in P$. Exchanging $u$
and $v$, we can affirm that%
\[
\left(  f-\alpha u^{\ast}\right)  ^{+}\in P\quad\text{if and only if\quad
}\left(  f-\alpha v^{\ast}\right)  ^{+}\in P,
\]
which gives the desired equality.
\end{proof}

\section{Multi-truncations and a spectrum}

Recall that the symbol $\mathcal{P}$ is used to denote the set of all proper
prime ideals of the vector lattice $L$. For any non-empty subset $\mathcal{Q}$
of $\mathcal{P}$ and any element $f$ of $L$, we denote by $\left[
\mathcal{Q}\right]  _{f}$ the set of all ideals in $\mathcal{Q}$ which omit
$f$, i.e.,%
\[
\left[  \mathcal{Q}\right]  _{f}=\left\{  P\in\mathcal{Q}:f\notin P\right\}
.
\]
The family of the sets $\left[  \mathcal{Q}\right]  _{f}$, where $f$ runs
through the positive cone $L^{+}$ of $L$, forms a base of the so-called
hull-kernel topology on $\mathcal{Q}$. Evidently, the hull-kernel topology on
$\mathcal{Q}$ coincides with the topology induced from the hull-kernel
topology on $\mathcal{P}$ (for a detailed study of the hull-kernel topology on
prime ideals of a vector lattice, the reader is encouraged to consult
\cite[Section 36]{LZ}).

\begin{quote}
\textsl{From now on, whenever a non-empty subset of }$\mathcal{P}$\textsl{\ is
considered, it is systematically equipped with its hull-kernel topology as we
just explained.}
\end{quote}

\noindent Now we go straight into the second fundamental definition of the paper.

\begin{definition}
\emph{We call a }multi-truncation \emph{on }$L$\emph{\ any family
}$\mathcal{T}$\emph{\ of truncations on }$L$\emph{\ such that}%
\[
f^{\ast}\wedge f^{\rtimes}=0\quad\text{\emph{for}}\emph{\ }\text{\emph{all}%
}\emph{\ }f\in L^{+}\emph{\ }\text{\emph{and}}\ \ast,\rtimes\in\mathcal{T}%
\text{ \emph{with}}\emph{\ }\ast\neq\rtimes.
\]

\end{definition}

It is a easy task to check that a family $\mathcal{T}$ of truncations on $L$
is a multi-truncation on $L$ if and only if%
\[
f^{\ast}\wedge g^{\rtimes}=0\quad\text{for all }f,g\in L^{+}\text{ and }%
\ast,\rtimes\in\mathcal{T}\text{ with }\ast\neq\rtimes.
\]
This necessary and sufficient condition will be used later without further mention.

The multi-truncation $\mathcal{T}$ is said to be \textsl{maximal} if it is not
strictly contained in another multi-truncation on $L$. The following
characterizes this special class of multi-truncations in the Archimedean case.
Recall first that the vector lattice $L$ is said to be \emph{Archimedean} if,
given $f,g\in L^{+}$, $nf\leq g$ for all $n\in\left\{  1,2,...\right\}  $,
then $f=0$. Obviously, $L$ is Archimedean if and only if%
\[
\inf\left\{  \frac{1}{n}f:n=1,2,...\right\}  =0\text{ for all }f\in L^{+}.
\]

\begin{lemma}
\label{maximal}If $L$ is Archimedean then a multi-truncation $\mathcal{T}$ is
maximal if and only if, for every $f\in L$ with $f>0$, there exists $\ast
\in\mathcal{T}$ such that $f^{\ast}>0$.
\end{lemma}

\begin{proof}
Assume that $\mathcal{T}$ is maximal and let $f\in L^{+}$ such that $f^{\ast
}=0$ for all $\ast\in\mathcal{T}$. We must prove that $f=0$. Arguing by
contradiction, we suppose that $f>0$. The formula%
\[
g^{\rtimes}=f\wedge g\text{ for all }g\in L
\]
defines a truncation $\rtimes$ on $L$. Indeed, the condition $\left(
\mathfrak{a}\right)  $ being clear, we prove the condition $\left(
\mathfrak{b}\right)  $. Let $g\in L^{+}$ such that $\left(  ng\right)
^{\rtimes}=ng$ for all $n\in\left\{  1,2,...\right\}  $. Thus, $0\leq ng\leq
f$ for all $n\in\left\{  1,2,...\right\}  $ from which we derive that $g=0$,
so $\rtimes$ is a truncation on $L$. Now, if $\ast\in\mathcal{T}$ and $g\in
L^{+}$ then%
\[
g^{\ast}\wedge g^{\rtimes}=g^{\ast}\wedge f\wedge g=f^{\ast}\wedge g=0.
\]
It follows that $\mathcal{T}\cup\left\{  \rtimes\right\}  $ is a
multi-truncation on $L$. By maximality, we derive that $\rtimes\in\mathcal{T}$
and so $f^{\rtimes}=0$. This yields that $0=f^{\rtimes}=f\wedge f=f$, a contradiction.

Conversely, assume that if $f\in L^{+}$ with $f^{\ast}=0$ for all $\ast
\in\mathcal{T}$, then $f=0$. We claim that $\mathcal{T}$ is maximal.
Otherwise, there exists a truncation $\rtimes$ on $L$ such that $\rtimes
\notin\mathcal{T}$ and $\mathcal{T}\cup\left\{  \rtimes\right\}  $ is a
multi-truncation on $L$. This shows that if $\ast\in\mathcal{T}$ and $f\in
L^{+}$ then%
\[
0\leq\left(  f^{\rtimes}\right)  ^{\ast}=\left(  f^{\rtimes}\right)  ^{\ast
}\wedge f=f^{\rtimes}\wedge f^{\ast}=0.
\]
Hence,%
\[
\left(  f^{\rtimes}\right)  ^{\ast}=0\text{ for all }\ast\in\mathcal{T},
\]
from which it follows that $f^{\rtimes}=0$. We derive that $\rtimes$ vanishes
on the whole $L^{+}$, contradicting the definition itself of a truncation.
\end{proof}

It should be noted in passing that a standard argument based on Zorn's Lemma
yields that any truncation on $L$ is contained in a maximal multi-truncation
on $L$.

\begin{quote}
\textsl{In the rest of the section, we fix a multi-truncation }$\mathcal{T}$
\textsl{on the vector lattice }$L$ \textsl{and a non-empty subset
}$\mathcal{Q}$ of $\mathcal{P}$.
\end{quote}

\noindent Define%
\[
\mathcal{Q}^{\ast}=%
{\displaystyle\bigcup_{f\in L^{+}}}
\left[  \mathcal{Q}\right]  _{f^{\ast}}\text{ for all }\ast\in\mathcal{T\quad
}\text{and}\mathcal{\quad Q}_{\mathcal{T}}=%
{\displaystyle\bigcup_{\ast\in\mathcal{T}}}
\mathcal{Q}^{\ast}.
\]
An alternative base of the hull-kernel topology on $\mathcal{Q}_{\mathcal{T}}$
is provided in what follows.

\begin{lemma}
\label{Partition}

\begin{enumerate}
\item[\emph{(i)}] $\left\{  \mathcal{Q}^{\ast}:\ast\in\mathcal{T}\right\}  $
is a partition of $\mathcal{Q}_{\mathcal{T}}$ into open-closed sets.

\item[\emph{(ii)}] The sets $\left[  \mathcal{Q}\right]  _{f^{\ast}}$, where
$\ast$ ranges over $\mathcal{T}$ and $f$ ranges over $L^{+}$, form a base for
$\mathcal{Q}_{\mathcal{T}}$.
\end{enumerate}
\end{lemma}

\begin{proof}
$\mathrm{(i)}$ It only takes a moment's thought to see that for each
$P\in\mathcal{Q}_{\mathcal{T}}$ there is exactly one truncation $\ast$ in
$\mathcal{T}$ such that $P\in\mathcal{Q}^{\ast}$. This means that the family
$\left\{  \mathcal{Q}^{\ast}:\ast\in\mathcal{T}\right\}  $ is, indeed, a
partition of $\mathcal{Q}_{\mathcal{T}}$. Moreover, if $\ast\in\mathcal{T}$
and $f\in L^{+}$ then%
\[
\left[  \mathcal{Q}_{\mathcal{T}}\right]  _{f^{\ast}}=\left[  \mathcal{Q}%
^{\ast}\right]  _{f^{\ast}}=\left[  \mathcal{Q}\right]  _{f^{\ast}}.
\]
It follows that $\left[  \mathcal{Q}\right]  _{f^{\ast}}$ is an open set in
$\mathcal{Q}_{\mathcal{T}}$ and so is $\mathcal{Q}^{\ast}$. We infer in
particular that the union $\underset{\rtimes\in\mathcal{T}\backslash\left\{
\ast\right\}  }{%
{\displaystyle\bigcup}
}\mathcal{Q}^{\rtimes}$ is again an open set in $\mathcal{Q}_{\mathcal{T}}$.
Thus, $\mathcal{Q}^{\ast}$ is a closed set in $\mathcal{Q}_{\mathcal{T}}$ and
$\mathrm{(i)}$ follows.

$\mathrm{(ii)}$ Choose $f\in L^{+}$ and $P\in\left[  \mathcal{Q}_{\mathcal{T}%
}\right]  _{f}$. So, there exists $\ast\in\mathcal{T}$ such that $P\in\left[
\mathcal{Q}^{\ast}\right]  _{f}$. In particular, $P\in\left[  \mathcal{Q}%
\right]  _{g^{\ast}}$ for some $g\in L^{+}$. Since $P$ is prime, we get%
\[
\left(  f\wedge g\right)  ^{\ast}=f^{\ast}\wedge g^{\ast}=f\wedge g^{\ast
}\notin P.
\]
Putting $h=f\wedge g$, we derive that%
\[
P\in\left[  \mathcal{Q}\right]  _{h^{\ast}}=\left[  \mathcal{Q}^{\ast}\right]
_{h^{\ast}}=\left[  \mathcal{Q}_{\mathcal{T}}\right]  _{h^{\ast}}.
\]
Moreover, if $Q\in\left[  \mathcal{Q}_{\mathcal{T}}\right]  _{h^{\ast}}$ then
$f\notin Q$ because $h^{\ast}\leq f$. Consequently, $Q\in\left[
\mathcal{Q}_{\mathcal{T}}\right]  _{f}$ and thus%
\[
P\in\left[  \mathcal{Q}_{\mathcal{T}}\right]  _{h^{\ast}}\subset\left[
\mathcal{Q}_{\mathcal{T}}\right]  _{f}.
\]
This leads to the assertion $\mathrm{(ii)}$ and completes the proof of the lemma.
\end{proof}

We are going to apply the previous results in a crucial particular case. It is
well known that for every $f\in L$ with $f\neq0$, there exists $P\in
\mathcal{P}$ which is maximal with respect to not containing $f$ (for the
proof, see \cite[Proposition 10.1]{D} or \cite[Theorem 33.5 ]{LZ}). Following
the terminology of \cite{BKW,D}, such an ideal $P$ is referred to as a
\textsl{value} of $f$. The (non-empty) set of all values of $f$ is denoted by
$\mathrm{Val}\left(  f\right)  $. To proceed our study, we need the following
fundamental lemma.

\begin{lemma}
\label{Value}Let $\ast$ be truncation on $L$ and $f,g\in L^{+}$. Then%
\[
\mathrm{Val}\left(  f^{\ast}\right)  \cap\left[  \mathcal{P}\right]
_{g^{\ast}}=\mathrm{Val}\left(  g^{\ast}\right)  \cap\left[  \mathcal{P}%
\right]  _{f^{\ast}}.
\]

\end{lemma}

\begin{proof}
Let $P\in\mathrm{Val}\left(  f^{\ast}\right)  $ and assume that $g^{\ast
}\notin P$. We claim that $P\in\mathrm{Val}\left(  g^{\ast}\right)  $. To this
end, we shall argue by contradiction supposing that $P\notin\mathrm{Val}%
\left(  g^{\ast}\right)  $. By \cite[Theorem 33.5]{LZ}, we can find
$Q\in\mathrm{Val}\left(  g^{\ast}\right)  $ such that $P\subset Q$. Because
$P\neq Q$ and $P\in\mathrm{Val}\left(  f^{\ast}\right)  $, we must have
$f^{\ast}\in Q$. Now, from Lemma \ref{BH} $\mathrm{(i)}$ it follows that%
\[
\left(  mf-\left(  mf\right)  ^{\ast}\right)  ^{\ast}\notin P\quad
\text{and\quad}\left(  ng-\left(  ng^{\ast}\right)  \right)  ^{\ast}\notin P
\]
hold for some $m,n\in\left\{  1,2,...\right\}  $. Using Lemma \ref{Ball}
$\mathrm{(ii)}$, we obtain%
\[
\left(  mf\right)  ^{\ast}-\left(  ng\right)  ^{\ast}\in P\subset Q.
\]
Furthermore, $\left(  mf\right)  ^{\ast}\in Q$ because $f^{\ast}\in Q$ (see
Lemma \ref{Ball} $\mathrm{(iii)}$). We may conclude that $\left(  ng\right)
^{\ast}\in Q$. But then $g^{\ast}\in Q$ which contradicts the condition
$Q\in\mathrm{Val}\left(  g^{\ast}\right)  $ and shows that $P\in
\mathrm{Val}\left(  g^{\ast}\right)  $. We infer that%
\[
\mathrm{Val}\left(  f^{\ast}\right)  \cap\left[  \mathcal{P}\right]
_{g^{\ast}}\subset\mathrm{Val}\left(  g^{\ast}\right)  \cap\left[
\mathcal{P}\right]  _{f^{\ast}}.
\]
We just swap $f$ and $g$ to get the desired equality.
\end{proof}

At this point, we set%
\[
\mathcal{V}=%
{\displaystyle\bigcup_{\ast\in\mathcal{T}}}
{\displaystyle\bigcup_{f\in L^{+}}}
\mathrm{Val}\left(  f^{\ast}\right)  .
\]
Observe that from Lemma \ref{Value} it follows quite easily that%
\[
\left[  \mathcal{V}\right]  _{f^{\ast}}=\mathrm{Val}\left(  f^{\ast}\right)
\text{ for all }f\in L^{+}.
\]
Accordingly,%
\[
\mathcal{V}^{\ast}=%
{\displaystyle\bigcup_{f\in L^{+}}}
\mathrm{Val}\left(  f^{\ast}\right)  \text{ for all }\ast\in\mathcal{T\quad
}\text{and\quad}\mathcal{V}=\mathcal{V}_{\mathcal{T}}.
\]
It turns out that the particular space $\mathcal{V}$, which we call the
\textsl{spectrum} of $L$ \textsl{with respect to the multi-truncation
}$\mathcal{T}$, enjoys remarkable topological properties.

\begin{lemma}
\label{Hausdorff}

\begin{enumerate}
\item[\emph{(i)}] $\mathcal{V}$ is a locally compact Hausdorff space.

\item[\emph{(ii)}] If $L$ is Archimedean and $\mathcal{T}$ is maximal then
$\mathcal{V}$ is dense in $\mathcal{P}$.
\end{enumerate}
\end{lemma}

\begin{proof}
$\mathrm{(i)}$ Let $P\in\mathcal{V}$ and $\ast\in\mathcal{T}$ such that
$P\in\mathcal{V}^{\ast}$. Hence, $P\in\mathrm{Val}\left(  f^{\ast}\right)  $
for some $f\in L^{+}$. We call \cite[Theorem 36.4 (i)]{LZ} to affirm that
$\mathrm{Val}\left(  f^{\ast}\right)  =\left[  \mathcal{V}\right]  _{f^{\ast}%
}$ is a compact neighborhood of $P$. This shows that $\mathcal{V}$ is locally
compact, as desired. Now, we claim that $\mathcal{V}$ is Hausdorff. To this
end, pick $P,Q\in\mathcal{V}$ with $P\neq Q$. Hence, there exist $\ast
,\rtimes\in\mathcal{T}$ and $f,g\in L^{+}$ such that%
\[
P\in\left[  \mathcal{V}\right]  _{f^{\ast}}=\mathrm{Val}\left(  f^{\ast
}\right)  \quad\text{and\quad}Q\in\left[  \mathcal{V}\right]  _{g^{\rtimes}%
}=\mathrm{Val}\left(  g^{\rtimes}\right)  .
\]
If $\ast\neq\rtimes$ then the open sets $\left[  \mathcal{V}\right]
_{f^{\ast}}$ and $\left[  \mathcal{V}\right]  _{g^{\rtimes}}$ are disjoint.
Now, suppose that $\ast=\rtimes$ and let $h=f+g$. Observe that $f^{\ast}\leq
h^{\ast}$ and so $h^{\ast}\notin P$. Analogously, we have $h^{\ast}\notin Q$.
In other words, $P,Q\in\left[  \mathcal{V}\right]  _{h^{\ast}}$ and so
$P,Q\in\mathrm{Val}\left(  h^{\ast}\right)  $ (where we use Lemma
\ref{Value}). Then, by maximality, neither $P\subset Q$ nor $Q\subset P$.
Therefore, there exists positive elements $u\in P$ and $v\in Q$ such that
$0<v\notin P$ and $0<u\notin Q$. Put%
\[
a=u-\left(  u\wedge v\right)  \quad\text{and\quad}b=v-\left(  u\wedge
v\right)  .
\]
Observe that $a,b>0$ and $a\wedge b=0$, so $\mathrm{Val}\left(  a^{\ast
}\right)  \cap\mathrm{Val}\left(  b^{\ast}\right)  =\emptyset$ (by primality).
Moreover, it is not hard to see that $b^{\ast}\notin P$ (because $b\notin P$)
and thus, again by Lemma \ref{Value}, we obtain $P\in\mathrm{Val}\left(
b^{\ast}\right)  =\left[  \mathcal{V}\right]  _{b^{\ast}}$. Similarly,
$Q\in\mathrm{Val}\left(  a^{\ast}\right)  =\left[  \mathcal{V}\right]
_{a^{\ast}}$, which shows that $\mathcal{V}$ is Hausdorff, as required.

$\mathrm{(ii)}$ Let $f\in L^{+}$ such that $f\in P$ for all $P\in\mathcal{V}$.
Arguing by contradiction, we assume that $0<f$. As $\mathcal{T}$ is maximal,
Lemma \ref{maximal} yields that $0<f^{\ast}$ for some $\ast\in\mathcal{T}$.
Therefore, we can find $P\in\mathrm{Val}\left(  f^{\ast}\right)
\subset\mathcal{V}$. But then $f\notin P$, which is an obvious contradiction.
We derive that $\underset{P\in\mathcal{V}}{%
{\displaystyle\bigcap}
}P=\left\{  0\right\}  $ and density follows from \cite[Theorem 36.1]{LZ}.
\end{proof}

\section{Multi-truncations and extended-real valued functions}

We shall keep the same notations of the previous section. On the other hand,
we denote by $\overline{\mathbb{R}}$ the two-point compactification of the
real line $\mathbb{R}$.

\begin{quote}
\textsl{As in the previous section, $\mathcal{T}$ is a multi-truncation on
}$L$ \textsl{and $\mathcal{Q}$ is a non-empty subset}\textit{\ }$\mathcal{Q}$
\textsl{of}\textit{\ }$\mathcal{P}$.
\end{quote}

\noindent According to the first assertion in Lemma \ref{Partition}, any
function on $\mathcal{Q}_{\mathcal{T}}$ can be defined by its respective
restrictions to the sets $\mathcal{Q}^{\ast}$, where $\ast$ runs through
$\mathcal{T}$. In so doing, for any $f\in L$, we define an extended-real
valued function $\widehat{f}:\mathcal{Q}_{\mathcal{T}}\rightarrow
\overline{\mathbb{R}}$ as follows. For every $P\in\mathcal{Q}_{\mathcal{T}}$,
there exists a unique $\ast\in\mathcal{T}$ such that $P\in\mathcal{Q}^{\ast}$.
Choose then an arbitrary element $u$ in the set%
\[
\pi^{\ast}\left(  P\right)  =\left\{  v\in L^{+}:\left(  v-v^{\ast}\right)
^{\ast}\notin P\right\}
\]
and put%
\[
\widehat{f}\left(  P\right)  =\inf\left\{  \alpha\in\mathbb{R}:\left(
f-\alpha u^{\ast}\right)  ^{+}\in P\right\}  \text{ for all }P\in
\mathcal{Q}^{\ast}%
\]
(with the usual agreements $\inf\mathbb{R}=-\infty$ and $\inf\emptyset=\infty
$). Contrary to what may be thought at first glance, Lemma \ref{BH}
$\mathrm{(ii)}$ guarantees us that the value of $\widehat{f}$ at the point $P$
does not depend on the choice of $u$ and depends only on $f$ and $P$.

Some of these functions have noteworthy behaviors.

\begin{lemma}
\label{One-Zero}Let $\ast\in\mathcal{T}$ and $P\in\mathcal{Q}^{\ast}$. Then

\begin{enumerate}
\item[\emph{(i)}] $\widehat{u^{\ast}}\left(  P\right)  =1$ for all $u\in
\pi^{\ast}\left(  P\right)  $,

\item[\emph{(ii)}] $\widehat{f}\left(  P\right)  =0$ for all $f\in P$, and

\item[\emph{(iii)}] $\widehat{f^{\rtimes}}\left(  P\right)  =0$ for all $f\in
L^{+}$ and $\rtimes\in\mathcal{T}\backslash\left\{  \ast\right\}  $.
\end{enumerate}
\end{lemma}

\begin{proof}
$\mathrm{(i)}$ Let $u\in\pi^{\ast}\left(  P\right)  $ and $\alpha<1$. Since%
\[
u^{\ast}\geq\left(  u-u^{\ast}\right)  ^{\ast}\notin P,
\]
we see that $u^{\ast}\notin P$ and so%
\[
\left(  u^{\ast}-\alpha u^{\ast}\right)  ^{+}=\left(  1-\alpha\right)
u^{\ast}\notin P.
\]
Moreover, if $\alpha=1$ then $\left(  u^{\ast}-\alpha u^{\ast}\right)
^{+}=0\in P$ from which we derive that%
\[
\widehat{u^{\ast}}\left(  P\right)  =\inf\left\{  \alpha\in\mathbb{R}:\left(
u^{\ast}-\alpha u^{\ast}\right)  ^{+}\in P\right\}  =1.
\]

$\mathrm{(ii)}$ On the other hand, if $f\in P$ and $\alpha<0$ then%
\[
0\leq-\alpha u^{\ast}\leq\left(  f-\alpha u^{\ast}\right)  ^{+}-f.
\]
Therefore, $\left(  f-\alpha u^{\ast}\right)  ^{+}\notin P$ because $-\alpha
u^{\ast}\notin P$. Furthermore, if $\alpha=0$ then $\left(  f-\alpha u^{\ast
}\right)  ^{+}=f^{+}\in P$. This yields that%
\[
\widehat{f}\left(  P\right)  =\inf\left\{  \alpha\in\mathbb{R}:\left(
f-\alpha u^{\ast}\right)  ^{+}\in P\right\}  =0.
\]

$\mathrm{(iii)}$ Let $f\in L^{+}$ and $\rtimes\in\mathcal{T}\backslash\left\{
\ast\right\}  $. Clearly, $g^{\ast}\notin P$ for some $g\in L^{+}$. Since
$\rtimes\neq\ast$, we get $f^{\rtimes}\wedge g^{\ast}=0$ and so $f^{\rtimes
}\in P$. The assertion $\mathrm{(ii)}$ leads to the desired conclusion.
\end{proof}

Now, we consider the set%
\[
\widehat{L}=\left\{  \widehat{f}:f\in L\right\}
\]
We say that $\widehat{L}$ \textsl{separates point from closed sets in
}$\mathcal{Q}_{\mathcal{T}}$ if for every closed set $\mathcal{F}$ in
$\mathcal{Q}_{\mathcal{T}}$ and $P\in\mathcal{Q}_{\mathcal{T}}$ with
$P\notin\mathcal{F}$, there exists some $f\in L^{+}$ such that $f\left(
P\right)  =1$ and $f\left(  Q\right)  =0$ for all $Q\in\mathcal{F}$ (see,
e.g., \cite[Page 60]{B1}). On the other hand, we denote by $1_{D}$ the
characteristic (also called indicator) function on any set $D$.

The following lemma plays an essential role in the proof of the main result of
this work.

\begin{lemma}
\label{Topology}

\begin{enumerate}
\item[\emph{(i)}] For every $f\in L$, the function $\widehat{f}$ is continuous
on $\mathcal{Q}_{\mathcal{T}}$.

\item[\emph{(ii)}] $\widehat{L}$ separates points form closed sets in
$\mathcal{Q}_{\mathcal{T}}$.

\item[\emph{(iii)}] $\widehat{f^{\ast}}=1_{\mathcal{Q}^{\ast}}\wedge f$ for
all $\ast\in\mathcal{T}$ and $f\in L^{+}$.
\end{enumerate}
\end{lemma}

\begin{proof}
$\mathrm{(i)}$ The following proof is, in part, inspired by the proof of
\cite[Theorem 5.1]{JK}. Let $f\in L$ and $P\in\mathcal{Q}_{\mathcal{T}}$.
There exists, by Lemma \ref{Partition} $\mathrm{(i)}$, a unique $\ast
\in\mathcal{T}$ such that $P\in\mathcal{Q}^{\ast}$. Select any $u\in\pi^{\ast
}\left(  P\right)  $, so%
\[
\widehat{f}\left(  P\right)  =\inf\left\{  \alpha\in\mathbb{R}:\left(
f-\alpha u^{\ast}\right)  ^{+}\in P\right\}  .
\]
Observe that $u^{\ast}\notin P$ because $\left(  u-u^{\ast}\right)  ^{\ast
}\notin P$ and $u^{\ast}\geq\left(  u-u^{\ast}\right)  ^{\ast}$. Thus,%
\[
\widehat{f}\left(  P\right)  =\sup\left\{  \alpha\in\mathbb{R}:\left(
f-\alpha u^{\ast}\right)  ^{-}\in P\right\}  .
\]
Indeed, take $\alpha,\beta\in\mathbb{R}$ such that%
\[
\left(  f-\alpha u^{\ast}\right)  ^{+},\left(  f-\beta u^{\ast}\right)
^{-}\in P.
\]
We claim that $\beta\leq\alpha$. Otherwise, we would have $\beta-\alpha>0$ and
so%
\begin{align*}
0  &  <\left(  \beta-\alpha\right)  u^{\ast}=\left(  f-\alpha u^{\ast}\right)
-\left(  f-\beta u^{\ast}\right) \\
&  \leq\left(  f-\alpha u^{\ast}\right)  ^{+}-\left(  f-\beta u^{\ast}\right)
^{-}\in P.
\end{align*}
This contradicts the fact that $u^{\ast}\notin P$.

Now, suppose that $\widehat{f}\left(  P\right)  =r\in\mathbb{R}$ and take
$\varepsilon\in\left(  0,\infty\right)  $. By the previous part, we have%
\[
\left(  f-\left(  r-\varepsilon\right)  u^{\ast}\right)  ^{+}\notin
P\quad\text{and}\quad\left(  f-\left(  r+\varepsilon\right)  u^{\ast}\right)
^{-}\notin P.
\]
Put%
\[
g=\left(  u-u^{\ast}\right)  ^{\ast}\wedge\left(  f-\left(  r-\varepsilon
\right)  u^{\ast}\right)  ^{+}\wedge\left(  f-\left(  r+\varepsilon\right)
u^{\ast}\right)  ^{-}%
\]
and notice that $0<g^{\ast}=g\notin P$ since $P$ is prime. This means that
$P\in\left[  \mathcal{Q}\right]  _{g^{\ast}}$ from which we infer that
$\left[  \mathcal{Q}\right]  _{g^{\ast}}$ is an open neighborhood of $P$ in
$\mathcal{Q}_{\mathcal{T}}$ (see Lemma \ref{Partition} $\mathrm{(ii)}$).
Moreover, if $Q\in\left[  \mathcal{Q}\right]  _{g^{\ast}}$ then $\left(
u-u^{\ast}\right)  ^{\ast}\notin Q$, so $u\in\pi^{\ast}\left(  Q\right)  $.
Thus, we can write%
\begin{align*}
\widehat{f}\left(  Q\right)   &  =\inf\left\{  \alpha\in\mathbb{R}:\left(
f-\alpha u^{\ast}\right)  ^{+}\in Q\right\} \\
&  =\sup\left\{  \alpha\in\mathbb{R}:\left(  f-\alpha u^{\ast}\right)  ^{-}\in
Q\right\}  .
\end{align*}
Furthermore,%
\[
\left(  f-\left(  r-\varepsilon\right)  u^{\ast}\right)  ^{+}\notin
Q\quad\text{and\quad}\left(  f-\left(  r+\varepsilon\right)  u^{\ast}\right)
^{-}\notin Q.
\]
We derive directly that%
\[
r-\varepsilon\leq\widehat{f}\left(  Q\right)  \leq r+\varepsilon,
\]
which yields that $\widehat{f}$ is continuous in $P$.

Secondly, assume that $\widehat{f}\left(  P\right)  =\infty$, that is,%
\[
\left\{  \alpha\in\mathbb{R}:\left(  f-\alpha u^{\ast}\right)  ^{+}\in
P\right\}  =\emptyset.
\]
Pick a real number $\lambda$ and define%
\[
g=\left(  u-u^{\ast}\right)  ^{\ast}\wedge\left(  f-\lambda u^{\ast}\right)
^{+}.
\]
Consequently, $g^{\ast}=g\notin P$ and then $\left[  \mathcal{Q}\right]
_{g^{\ast}}$ is an open neighborhood of $P$ in $\mathcal{Q}_{\mathcal{T}}$.
Given $Q\in\left[  \mathcal{Q}\right]  _{g^{\ast}}$, we see that $\left(
f-\lambda u^{\ast}\right)  ^{+}\notin Q$ and $\left(  u-u^{\ast}\right)
^{\ast}\notin Q$. It follows that $\widehat{f}\left(  Q\right)  \geq\lambda$,
meaning that $\widehat{f}$ is again continuous in $P$.

Finally, suppose that $\widehat{f}\left(  P\right)  =-\infty$. In other words,%
\[
\left\{  \alpha\in\mathbb{R}:\left(  f-\alpha u^{\ast}\right)  ^{+}\in
P\right\}  =\mathbb{R}%
\]
and, by the first case,%
\[
\left\{  \alpha\in\mathbb{R}:\left(  \left(  -f\right)  -\alpha u^{\ast
}\right)  ^{+}\in P\right\}  =\emptyset.
\]
Applying the previous case to $-f$, we deduce that $\widehat{-f}$ is
continuous on $P$. But then $\widehat{f}$ is also continuous in $P$ because
$\widehat{-f}=-\widehat{f}$ (see Theorem 44.3 $\mathrm{(i)}$ in \cite{LZ}).

$\mathrm{(ii)}$ Let $\mathcal{F}$ be a closed set in $\mathcal{Q}%
_{\mathcal{T}}$ and $Q\in\mathcal{Q}_{\mathcal{T}}$ with $Q\notin\mathcal{F}$.
Using Lemma \ref{Partition} $\mathrm{(ii)}$, there exists $\ast\in\mathcal{T}$
and $f\in L^{+}$%
\[
Q\in\left[  \mathcal{Q}\right]  _{f^{\ast}}\quad\text{and}\quad\mathcal{F}%
\cap\left[  \mathcal{Q}\right]  _{f^{\ast}}=\emptyset.
\]
By Lemma \ref{BH} $\mathrm{(i)}$, we can find $n\in\left\{  1,2,...\right\}  $
such that $nf\in\pi^{\ast}\left(  Q\right)  $. Put $u=nf-\left(  nf\right)
^{\ast}$ and observe that Lemma \ref{One-Zero} $\mathrm{(i)}$ yields that
$\widehat{u^{\ast}}\left(  Q\right)  =1$. Moreover, if $P\in\mathcal{F}$ then
$f^{\ast}\in P$ and so $u^{\ast}\in P$ (again by Lemma \ref{Ball}
$\mathrm{(iii)}$). Using Lemma \ref{One-Zero} $\mathrm{(ii)}$, we obtain
$\widehat{u^{\ast}}\left(  P\right)  =0$ and the required separation property.

$\mathrm{(iii)}$ Let $f\in L^{+}$ and $\ast\in\mathcal{T}$. If $P\in
\mathcal{Q}_{T}$ with $P\notin\mathcal{Q}^{\ast}$ then $f^{\ast}\in P$ and so%
\[
\widehat{f^{\ast}}\left(  P\right)  =0=\left(  1_{\mathcal{Q}^{\ast}}\wedge
f\right)  \left(  P\right)  .
\]
(where we use Lemma \ref{One-Zero} $\mathrm{(ii)}$). Now, if $P\in
\mathcal{Q}^{\ast}$ and $u\in\pi^{\ast}\left(  P\right)  $ then, by Lemma
\ref{One-Zero} $\mathrm{(i)}$, we get $\widehat{u^{\ast}}\left(  P\right)
=1$. Furthermore, Lemma \ref{Ball} $\mathrm{(i)}$ shows that $f^{\ast}-f\wedge
u^{\ast}\in P$. This together with Lemma \ref{One-Zero} yields that
$\widehat{f^{\ast}-f\wedge u^{\ast}}\left(  P\right)  =0$. Taking into account
elementary identities in \cite[Theorem 44.3]{LZ}, we get
\[
\widehat{f^{\ast}}\left(  P\right)  -\widehat{f}\left(  P\right)
\wedge1=\widehat{f^{\ast}-f\wedge u^{\ast}}\left(  P\right)  =0.
\]
Thus,%
\[
\widehat{f^{\ast}}\left(  P\right)  =\widehat{f}\left(  P\right)
\wedge1=\left(  1_{\mathcal{Q}^{\ast}}\wedge\widehat{f}\right)  \left(
P\right)  ,
\]
which gives $\mathrm{(iii)}$ and completes the proof of the lemma.
\end{proof}

Following \cite{L} and \cite{LZ}, we denote by $C^{\infty}\left(  X\right)  $
the set of all continuous functions $\varphi$ from a topological space $X$
into $\overline{\mathbb{R}}$ such that the set $\left\{  x\in X:\left\vert
\varphi\left(  x\right)  \right\vert \neq\infty\right\}  $ is dense in $X$.
Functions in $C^{\infty}\left(  X\right)  $ are usually called
\textsl{almost-finite extended-real valued continuous functions }on $X$. It is
well known that $C^{\infty}\left(  X\right)  $ is a lattice with respect to
the pointwise ordering and if $\varphi\in C^{\infty}\left(  X\right)  $ then
$r\varphi\in C^{\infty}\left(  X\right)  $ for all $r\in\mathbb{R}$. However,
$C^{\infty}\left(  X\right)  $ need not be a vector lattice, simply because if
$\varphi,\psi\in C^{\infty}\left(  X\right)  $ there may not exist a function
$\phi$ in $C^{\infty}\left(  X\right)  $ such that $\phi=\varphi+\psi$. In
spite of this pathological behavior, we shall allow ourselves to call a
\textsl{vector sublattice} of $C^{\infty}\left(  X\right)  $ any sublattice of
$C^{\infty}\left(  X\right)  $ which is closed under pointwise addition and
scalar multiplication.

\begin{lemma}
\label{sublattice}If $L$ is Archimedean and $\mathcal{Q}$ is dense in
$\mathcal{P}$, then $\widehat{L}$ is a vector sublattice of $C^{\infty}\left(
\mathcal{Q}_{\mathcal{T}}\right)  $.
\end{lemma}

\begin{proof}
Taking into consideration Lemma \ref{Topology} $\mathrm{(i)}$ and Theorem 44.3
in \cite{LZ} (see also \cite[Page 83]{L}), it suffices to show that the open
set%
\[
\left\{  P\in\mathcal{Q}_{\mathcal{T}}:\left\vert \widehat{f}\right\vert
\left(  P\right)  \neq\infty\right\}
\]
is dense in $\mathcal{Q}_{\mathcal{T}}$. Otherwise, the interior $\Omega$ of
the set%
\[
\left\{  P\in\mathcal{Q}_{\mathcal{T}}:\left\vert \widehat{f}\right\vert
\left(  P\right)  =\infty\right\}
\]
is non-empty. Using Lemma \ref{Partition} $\mathrm{(ii)}$, there exist
$\ast\in\mathcal{T}$ and $g\in L$ such that $g^{\ast}>0$ and%
\[
\left[  \mathcal{Q}\right]  _{g^{\ast}}=\left\{  P\in\mathcal{Q}:g^{\ast
}\notin P\right\}  \subset\Omega.
\]
Let $P\in\left[  \mathcal{Q}\right]  _{g^{\ast}}$ and observe that
$\widehat{\left\vert f\right\vert _{\ast}}\left(  P\right)  =\left\vert
\widehat{f}_{\ast}\right\vert \left(  P\right)  =\infty$ (where we use \cite[
Theorem 44.3 $\mathrm{(iii)}$]{LZ} for the first equality). Moreover, we know
that there exists $p\in\left\{  1,2,...\right\}  $ such that $pg\in\pi^{\ast
}\left(  P\right)  $ (see Lemma \ref{BH} $\mathrm{(i)}$). Putting $u=pg$, we
obtain%
\[
\inf\left\{  \alpha\in\mathbb{R}:\left(  \left\vert f\right\vert -\alpha
u^{\ast}\right)  ^{+}\in P\right\}  =\left\vert \widehat{f}\right\vert \left(
P\right)  =\infty.
\]
We quickly derive that%
\[
(nu^{\ast}-\left\vert f\right\vert )^{+}\in P\text{ for all }n\in\left\{
1,2,...\right\}  .
\]
Consequently,%
\[
(nu^{\ast}-\left\vert f\right\vert )^{+}\in%
{\displaystyle\bigcap_{g^{\ast}\notin P\in\mathcal{Q}}}
P\text{ for all }n\in\left\{  1,2,...\right\}  .
\]
However, $\mathcal{Q}$ is dense in $\mathcal{P}$ and so $\underset
{P\in\mathcal{Q}}{%
{\displaystyle\bigcap}
}P=\left\{  0\right\}  $ (see \cite[Theorem 36.1]{LZ}). Therefore, by
Corollary 35.4 in \cite{LZ}, we get%
\[
(u^{\ast}-\frac{1}{n}\left\vert f\right\vert )^{+}\wedge g^{\ast}=0\text{ for
all }n\in\left\{  1,2,...\right\}  .
\]
As $L$ is Archimedean, we infer that%
\[
u^{\ast}\wedge g^{\ast}=\sup\left\{  \left(  u^{\ast}-\frac{1}{n}\left\vert
f\right\vert \right)  ^{+}:n\in\left\{  1,2,...\right\}  \right\}  \wedge
g^{\ast}=0.
\]
This leads directly to the contradiction%
\[
g^{\ast}=g^{\ast}\wedge pg=g^{\ast}\wedge\left(  pg\right)  ^{\ast}=g^{\ast
}\wedge u^{\ast}=0.
\]
Accordingly, the open set $\left\{  P\in\mathcal{Q}_{\mathcal{T}}:\left\vert
\widehat{f}\right\vert \left(  P\right)  \neq\infty\right\}  $ is dense in
$\mathcal{Q}_{\mathcal{T}}$, which proves that $\widehat{f}\in C^{\infty
}\left(  \mathcal{Q}_{\mathcal{T}}\right)  $, as desired.
\end{proof}

Recall from the previous section that%
\[
\mathcal{V}=%
{\displaystyle\bigcup_{\ast\in\mathcal{T}}}
{\displaystyle\bigcup_{f\in L^{+}}}
\mathrm{Val}\left(  f^{\ast}\right)  .
\]
We arrive to the last lemma of this section.

\begin{lemma}
\label{last}If $\ast\in\mathcal{T}$ and $f>0$ in $L$, then $\left\{
P\in\mathcal{V}^{\ast}:\widehat{f^{\ast}}\left(  P\right)  =1\right\}  $ is a
compact set in $\mathcal{V}$.
\end{lemma}

\begin{proof}
It's easy to see that the set in question, say $\mathcal{K}$, is closed in
$\mathcal{V}$. Now take $P\in\mathcal{K}$ and use Lemma \ref{Value} to see
that%
\[
P\in\mathrm{Val}\left(  g^{\ast}\right)  \cap\left[  \mathcal{P}\right]
_{f^{\ast}}=\mathrm{Val}\left(  f^{\ast}\right)  \cap\left[  \mathcal{P}%
\right]  _{g^{\ast}}%
\]
holds for some $g\in L^{+}$. This yields that $\mathcal{K\subset}%
\mathrm{Val}\left(  f^{\ast}\right)  $ which is compact (see \cite[Theorem
36.4 (i)]{LZ}).
\end{proof}

\section{The main result and its applications}

We have gathered along the previous sections all the ingredients we need to
prove the following general representation theorem.

\begin{theorem}
\label{Johnson-Kist}Let $L$ be an Archimedean vector lattice, $\mathcal{T}$ be
a maximal multi-truncation, and $\mathcal{Q}$ be a dense subset of
$\mathcal{P}$. Then the map $\Lambda:L\rightarrow C^{\infty}\left(
\mathcal{Q}_{\mathcal{T}}\right)  $ defined by%
\[
\Lambda\left(  f\right)  =\widehat{f}\text{\quad for all }f\in L
\]
is a one-to-one lattice homomorphism such that%
\[
\Lambda\left(  f^{\ast}\right)  =1_{\mathcal{Q}^{\ast}}\wedge\Lambda\left(
f\right)  \text{\quad for all }\ast\in\mathcal{T}\text{ and }f\in L^{+}.
\]

\end{theorem}

\begin{proof}
In view of Lemmas \ref{Topology}-\ref{sublattice}, and Theorem 44.3 in
\cite{LZ}, it only remains for us to show that $\Lambda$ is one-to-one. Let
$f\in L$ such that $\widehat{f}\left(  P\right)  =0$ for all $P\in
\mathcal{Q}_{\mathcal{T}}$. We claim that $f=0$. To this end, we can assume
that $f\in L^{+}$. Suppose by contradiction that $f>0$. The rest of the proof
is somewhat reminiscent of the proof of Lemma \ref{sublattice}. By Lemma
\ref{maximal}, there exists $\ast\in\mathcal{T}$ such that $f^{\ast}>0$. Using
Lemma \ref{BH} $\mathrm{(i)}$, we can choose $n\in\left\{  1,2,...\right\}  $
such that $nf\in\pi^{\ast}\left(  P\right)  $ for any $P\in\left[
\mathcal{Q}\right]  _{f^{\ast}}$. Put $u=nf$ and take $P\in\left[
\mathcal{Q}\right]  _{f^{\ast}}$. Thus,%
\[
0=\widehat{f}\left(  P\right)  =\inf\left\{  \alpha\in\mathbb{R}:\left(
f-\alpha u^{\ast}\right)  ^{+}\in P\right\}  .
\]
Therefore,%
\[
\left(  f-\frac{1}{n}u^{\ast}\right)  ^{+}\in P\text{ for all }n\in\left\{
1,2,...\right\}  .
\]
We derive that%
\[
\left(  f-\frac{1}{n}u^{\ast}\right)  ^{+}\in%
{\displaystyle\bigcap_{P\in\left[  \mathcal{Q}\right]  _{f^{\ast}}}}
P\text{ for all }n\in\left\{  1,2,...\right\}  .
\]
Since $\underset{P\in\mathcal{Q}}{%
{\displaystyle\bigcap}
}P=\left\{  0\right\}  $ (by density of $\mathcal{Q}$ in $\mathcal{P}$), we
can call \cite[Corollary 35.4]{LZ} to write%
\[
\left(  f-\frac{1}{n}u^{\ast}\right)  ^{+}\wedge f^{\ast}=0\text{ for all
}n\in\left\{  1,2,...\right\}  .
\]
Consequently,%
\[
f^{\ast}=f\wedge f^{\ast}=\inf\left\{  \left(  f-\frac{1}{n}u^{\ast}\right)
^{+}:n\in\left\{  1,2,...\right\}  \right\}  \wedge f^{\ast}=0
\]
(remember here that $L$ is Archimedean). This contradiction completes the
proof of the theorem.
\end{proof}

The maximal orthogonal set (also called a maximal disjoint system) in Theorem
5.1 in \cite{JK} (see also Theorem 44.4 in \cite{LZ}) gives raise to a maximal
multi-truncation in an obvious way. This shows directly that Theorem
\ref{Johnson-Kist} is a direct generalization of the classical Johnson-Kist
Representation Theorem. The following corollary is also a consequence of
Theorem \ref{Johnson-Kist} combined with Lemma \ref{Hausdorff} $\mathrm{(ii)}$.

\begin{corollary}
\label{pre-main}Let $L$ be an Archimedean vector lattice and $\mathcal{T}$ be
a maximal multi-truncation. Then the map $\Lambda:L\rightarrow C^{\infty
}\left(  \mathcal{V}\right)  $ defined by%
\[
\Lambda\left(  f\right)  =\widehat{f}\text{\quad for all }f\in L
\]
is a one-to-one lattice homomorphism such that%
\[
\Lambda\left(  f^{\ast}\right)  =1_{\mathcal{V}^{\ast}}\wedge\Lambda\left(
f\right)  \text{\quad for all }\ast\in\mathcal{T}\text{ and }f\in L^{+}.
\]

\end{corollary}

We are now about to state and prove the main theorem of this research. For, it
could be helpful to label the following definition.

\begin{definition}
\emph{A vector lattice along with a truncation (in the sense of Definition
\ref{trunc}) is called a }truncated vector lattice.
\end{definition}

We proceed to our main result.

\begin{theorem}
\label{main}For any Archimedean truncated vector lattice $L$, there exists a
locally compact Hausdorff space $\mathcal{X}$ such that

\begin{enumerate}
\item[\emph{(i)}] $L$ is \emph{(}lattice isomorphic with\emph{)} a vector
sublattice of $C^{\infty}\left(  \mathcal{X}\right)  $,

\item[\emph{(ii)}] $L$ separates points from closed sets in $\mathcal{X}$,

\item[\emph{(iii)}] There exists an open-closed set $\mathcal{Y}$ in
$\mathcal{X}$ such that
\[
f^{\ast}=1_{\mathcal{Y}}\wedge f\text{ for all }f\in L^{+}.
\]

\item[\emph{(iv)}] $L$ vanishes nowhere on $\mathcal{X}$ \emph{(}i.e., for
every $x\in\mathcal{X}$ there exists $f\in L$ such that $f\left(  x\right)
\neq0$\emph{).}

\item[\emph{(v)}] For all $f\in L$ and $\varepsilon\in\left(  0,\infty\right)
$, the set $\left\{  y\in\mathcal{Y}:\left\vert f\left(  y\right)  \right\vert
\geq\varepsilon\right\}  $ is compact.
\end{enumerate}
\end{theorem}

\begin{proof}
Let $\ast$ denote the truncation on $L$. As noticed somewhere before, a
standard argument based on the Zorn's Lemma shows that there exists a maximal
multi-truncation $\mathcal{T}$ on $L$ containing $\ast$. Put $\mathcal{X}%
=\mathcal{V}_{\mathcal{T}}=\mathcal{V}$ and $\mathcal{Y}=\mathcal{V}^{\ast}$.
Hence, Corollary \ref{pre-main} gives the assertion $\mathrm{(i)}$ and
$\mathrm{(iii)}$. Furthermore, the assertions $\mathrm{(ii)}$ and
$\mathrm{(iv)}$ follow directly from Lemma \ref{Topology} $\mathrm{(ii)}$ and
Lemma \ref{One-Zero}, respectively$\mathrm{.}$ It remains to establish the
assertion $\mathrm{(v)}$. Choose $f\in L$ and $\varepsilon\in\left(
0,\infty\right)  $. By Lemma \ref{last}, the set%
\[
\mathcal{K}=\left\{  P\in\mathcal{V}^{\ast}:\widehat{\left(  \frac
{1}{\varepsilon}f\right)  ^{\ast}}\left(  P\right)  =1\right\}
\]
is a compact set. Observe now that%
\[
\mathcal{K}=\left\{  y\in\mathcal{Y}:\left(  1_{\mathcal{Y}}\wedge\frac
{1}{\varepsilon}\left\vert f\right\vert \right)  \left(  y\right)  =1\right\}
=\left\{  y\in\mathcal{Y}:\left\vert f\left(  y\right)  \right\vert
\geq\varepsilon\right\}  .
\]
This completes the proof of the theorem.
\end{proof}

Next, we show how we can apply Theorem \ref{main} to get two representation
theorems from the existing literature. To do this, we first recall the
following definition from the introduction.

\begin{definition}
\emph{The truncation}\textsl{ }$\ast$ \emph{on the}\textsl{ }$L$ \emph{is said
to be }weak\emph{\ if }$f=0$ \emph{provided }$f\in L^{+}$ \emph{and }$f^{\ast
}=0$\emph{. In this situation, we call }$L$\emph{ a}\textsl{ }weakly truncated
vector lattice\textsl{.}
\end{definition}

The main part of the following representation theorem is originally due to
Ball (see Theorem 5.3.6 in \cite{B1}). Here, we furnish an alternative way to
get the result.

\begin{theorem}
\label{Ball Rep}If $L$ is an Archimedean weakly truncated vector lattice, then
there exists a locally compact Hausdorff space $\mathcal{X}$ such that

\begin{enumerate}
\item[\emph{(i)}] $L$ is \emph{(}lattice isomorphic with\emph{)} a vector
lattice of functions in $C^{\infty}\left(  \mathcal{X}\right)  $,

\item[\emph{(ii)}] $L$ separates points from closed sets in $\mathcal{X}$,

\item[\emph{(iii)}] $f^{\ast}=1\wedge f$ for all $f\in L^{+}$,

\item[\emph{(iv)}] $L$ vanishes nowhere on $\mathcal{X}$, and

\item[\emph{(v)}] Any $f\in L$ vanishes at infinity.
\end{enumerate}
\end{theorem}

\begin{proof}
By Theorem \ref{main}, there exists a locally compact Hausdorff space
$\mathcal{X}$ such that the conditions $\mathrm{(i)}$,\textrm{\ (ii)}, and
$\mathrm{(iv)}$ are verified. Also, by $\mathrm{(iii)}$ in the same theorem,
there is an open-closed set $\mathcal{Y}$ of $\mathcal{X}$ such that $f^{\ast
}=1_{\mathcal{Y}}\wedge f$ for all $f\in L^{+}$. We claim that $\mathcal{Y}%
=\mathcal{X}$. Otherwise, there exists $x\in\mathcal{X}$ with $x\notin
\mathcal{Y}$. Since $\mathcal{Y}$ is a closed set in $\mathcal{X}$ and $L$
separates points from closed sets in $\mathcal{X}$, there exists $f\in L^{+}$
such that $f\left(  x\right)  =1$ and $f\left(  y\right)  =0$ for all
$y\in\mathcal{Y}$. But then $f^{\ast}=1_{\mathcal{Y}}\wedge f=0$, while
$f\neq0$. This contradicts the fact that the truncation on $L$ is weak and
leads to $\mathrm{(iii)}$. This together with Theorem \ref{main}
$\mathrm{(v)}$ proves the last assertion of the theorem.
\end{proof}

We conclude the paper with a representation theorem for strongly truncated
vector lattices. Although this theorem is an application of the main result of
this paper, we shall use also a Stone-Weierstrass type approximation theorem
recently obtained in \cite{BH}. First, let's introduce strongly truncated
vector lattices.

\begin{definition}
\emph{The truncation }$\ast$\emph{ on }$L$ \emph{is said to be }%
strong\emph{\ if}\textsl{\ }\emph{for every }$f\in L^{+}$\emph{\ there exists
}$n\in\left\{  1,2,...\right\}  $\emph{\ such that }$\left(  nf\right)
^{\ast}=nf$\emph{. In this case, we call }$L$\emph{ a }strongly truncated
vector lattice\textsl{.}
\end{definition}

Recall that if $\mathcal{X}$ is a locally compact Hausdorff space, then
$C_{0}\left(  \mathcal{X}\right)  $ denotes the Banach lattice of all
real-valued continuous functions on $\mathcal{X}$ vanishing at infinity.
Hence, if $L$ is a vector sublattice of $C_{0}\left(  \mathcal{X}\right)  $
such that $1\wedge f\in L$ for all $f\in L$, then $L$ is uniformly dense in
$C_{0}\left(  \mathcal{X}\right)  $ if and only if $L$ vanishes nowhere and
separates the points of $\mathcal{X}$ (for the proof, see Lemma 4.1 in
\cite{BH}).

\begin{theorem}
If $L$ is a strongly truncated vector lattice then there exists a locally
compact Hausdorff space $\mathcal{X}$ such that

\begin{enumerate}
\item[\emph{(i)}] $L$ is \emph{(}lattice isomorphic with\emph{)} a uniformly
dense vector sublattice of $C_{0}\left(  \mathcal{X}\right)  $,

\item[\emph{(ii)}] $L$ separates points from closed sets in $\mathcal{X}$, and

\item[\emph{(iii)}] $f^{\ast}=1\wedge f$ for all $f\in L^{+}.$
\end{enumerate}
\end{theorem}

\begin{proof}
First, let's see that $L$ is Archimedean. Pick $f,g\in L^{+}$ such that
$nf\leq g$ for all $n\in\left\{  1,2,...\right\}  $. Since the truncation is
strong, we can find $p\in\left\{  1,2,...\right\}  $ such that $\left(
pg\right)  ^{\ast}=pg$. Therefore, if $n\in\left\{  1,2,...\right\}  $ then
$0\leq nf\leq pg=\left(  pg\right)  ^{\ast}$ from which we derive that%
\[
\left(  nf\right)  ^{\ast}=nf\text{ for all }n\in\left\{  1,2,...\right\}  .
\]
This means that $f=0$ and so $L$ is Archimedean, as desired. Moreover, it is
readily checked that any strongly truncated vector lattice is a weakly
truncated vector lattice. So, by Theorem \ref{Ball Rep}, we can say that $L$
is (lattice isomorphic with) a vector sublattice of $C^{\infty}\left(
\mathcal{X}\right)  $ for some locally compact Hausdorff space $\mathcal{X}$
such that $\mathrm{(ii)}$ and $\mathrm{(iii)}$ hold. Moreover, the fact that
the truncation is strong together with the condition $\mathrm{(iii)}$ yields
that all \textquotedblleft functions\textquotedblright\ in $L$ are bounded and
so, using Theorem \ref{Ball Rep} $\mathrm{(v)}$, we derive that $L$ is a
vector sublattice of $C_{0}\left(  \mathcal{X}\right)  $. Now, observe that
that $L$ vanishes nowhere and separates the points of $\mathcal{X}$ (where we
use Theorem \ref{Ball Rep} $\mathrm{(ii)}$ and $\mathrm{(iv)}$). The
aforementioned Stone-Weierstrass Approximation Type Theorem ends the proof.
\end{proof}

\end{document}